\documentclass[11pt]{amsart}
\usepackage{amssymb}
\usepackage{palatino}
\usepackage{tikz-cd}
\input amssym.def

\usepackage{amsmath, amsfonts}
\usepackage{amssymb}
\usepackage{amscd}
\usepackage[mathscr]{eucal}
\usepackage{palatino}
\newfont{\cyrr}{wncyr10}

\newcommand{\thmref}[1]{Theorem~\ref{#1}}
\newcommand{\corref}[1]{Corollary~\ref{#1}}

\newcommand{\lemref}[1]{Lemma~\ref{#1}}

\newtheorem{thm}{Theorem}

\newtheorem{lem}[thm]{Lemma}
\newtheorem{cor}[thm]{Corollary}

\newtheorem{rmk}{Remark}[section]

\def\GL{{\rm GL}}
\def\G{{\rm G}}

\def\({\left(}
\def\){\right)}
\def\[{\left[}
\def\]{\right]}
\def\N{\mathbb{N}}
\def\Z{\mathbb{Z}}
\def\R{\mathbb{R}}
\def\C{\mathbb{C}}
\def\H{\mathbb{H}}
\def\Q{\mathbb{Q}}
\def\cN{\mathcal{N}}
\def\K{\mathbb{K}}
\def\L{\mathbb{L}}
\def\cO{\mathcal{O}}
\def\cP{\mathcal{P}}

\def\fn{\mathfrak{n}}

\def\sS{\mathscr{S}}
\def\sA{\mathscr{A}}
\def\sP{\mathscr{P}}
\def\div{~|~}
\def\a{\alpha}
\def\b{\beta}
\def\d{\delta}
\def\z{\zeta}

\def\fP{\mathfrak{P}}
\def\fp{\mathfrak{p}}

\def\w{\omega}
\def\e{\epsilon}

\newcommand{\tx}{\text}
\renewcommand{\mod}{{\, \rm mod \, }}

\usepackage{xcolor}
\usepackage{lipsum}

\let\OLDthebibliography\thebibliography
\renewcommand\thebibliography[1]{
	\OLDthebibliography{#1}
	\setlength{\parskip}{0pt}
	\setlength{\itemsep}{0pt plus 0.01ex}
}

\title{On a non-Archimedean analogue of a question of Atkin and Serre}

\author{Yuri Bilu, Sanoli Gun and Sunil Naik}

\address{Yuri ~F.~Bilu \newline
Université de Bordeaux and CNRS,
Institut de Mathématiques de Bordeaux UMR 5251,
33405, \newline
Talence, 
France.
 \newline
 \vspace{0.1cm}
Sanoli Gun and Sunil L Naik \newline
The Institute of Mathematical Sciences, 
A CI of Homi Bhabha National Institute, 
CIT Campus, Taramani, 
Chennai 600 113, 
India.}
\email{yuri@math.u-bordeaux.fr}
\email{sanoli@imsc.res.in}
\email{sunilnaik@imsc.res.in}

\begin{document}

\hfuzz 5pt

\subjclass[2010]{11F11, 11F30, 11F80, 11N56, 11N05, 11N36, 11R45}

\keywords{Largest prime factor, Fourier coefficients of Hecke eigenforms}

\begin{abstract}
In this article, we investigate a non-Archimedean analogue of a question 
of Atkin and Serre. More precisely, we derive lower bounds for the largest 
prime factor of non-zero Fourier coefficients of non-CM normalized 
cuspidal Hecke eigenforms of even weight $k \geq 2$, level $N$
with integer Fourier coefficients. In particular, we show that for such a 
form $f$ and for any real number $\e>0$, the largest prime factor 
of the $p$-th Fourier coefficient $a_f(p)$ of $f$, denoted by $P( a_f(p))$, 
satisfies
$$
P(a_f(p)) >  (\log p)^{1/8}(\log\log p)^{3/8 -\e}
$$
for almost all primes $p$. This improves on earlier bounds. 
We also investigate a number field analogue of a recent result of 
Bennett, Gherga, Patel and Siksek about the
largest prime factor of $a_f(p^m)$ for $m \ge 2$.
\end{abstract}

\maketitle

\section{Introduction and Statements of Results}
Throughout the article, let $p, q, \ell$ denote rational primes, 
$\H = \{ z \in \C : \Im(z) >0 \}$ the upper half-plane and $k \geq 2$ an even integer. 
Also let $f$ be a normalized cuspidal Hecke eigenform of 
even weight $k \geq 2$ for $\Gamma_0(N)$
with trivial character. The Fourier expansion of $f$ at infinity is given by
$$
f(z) = \sum_{n \geq 1} a_f(n) q^n~,
$$
where $q= e^{2\pi i z}$ and $z \in \H$.
It is well known that $a_f(n)$'s are real algebraic integers 
and $\K_f = \Q(\{a_f(n): n \in \N \})$ is a number field (see \cite{Sh}).
Serre (see \cite[Eq $4.11_k$]{JPS}), appealing to probabilistic considerations,
asked whether for any $\e >0$,
$$
|a_f(p)| \gg_{\e} p^{(k-3)/2 -\e}
$$
is true for a non-CM normalized Hecke eigenform $f$ of weight $k \geq 4$? 
Serre also mentioned that it was suggested to him by Atkin.  From now on, we shall refer to
it as Atkin-Serre question. 

In the present article, we consider a ``non-Archimedean'' version of this question, namely, 
what can one say about the largest prime factor of $a_f(p)$? 

For an integer $n$, let $P(n)$ denote the largest prime factor of $n$ with 
the convention that $P(0) = P(\pm 1) =1$. Also let us fix few notions of densities. 
For a subset~$S$ of primes, we shall define the lower and the upper densities of $S$ to be
$$
\liminf_{x \rightarrow \infty} \frac{\#\{p \leq x: p \in S\}}{\pi(x)}
\phantom{m}\text{and}\phantom{m}
\limsup_{x \rightarrow \infty} \frac{\#\{p \leq x: p \in S\}}{\pi(x)}
$$
respectively. Here $\pi(x)$ denotes the number of rational primes less than or equal to $x$.

If both upper and lower density of a subset $S$ of primes are equal, say to $\mathscr{D}$, 
we say that $S$ has density $\mathscr{D}$. 
We say a property $A$ holds for almost all primes if the set 
$$
\{p  :  p \text{ has property } A  \}
$$
has density one. 

Suppose that $f$ is a non-CM form with rational integer Fourier coefficients. It follows from the
work of Murty, Murty and Saradha \cite{MMS} that for any $\e>0$, we have
$$
P(a_f(p)) ~>~  e^{(\log\log p)^{1-\e}}
$$
for almost all primes $p$. Let $\tau$ denote the Ramanujan tau function defined by
\begin{equation*}
\sum_{n=1}^{\infty} \tau(n) q^n = q \prod_{n=1}^{\infty} (1-q^n)^{24}.
\end{equation*}
In \cite{LS}, Luca and Shparlinski proved that the inequality 
\begin{equation}\label{Neq1}
P(\tau(p) \tau(p^2)) > (\log p)^{\frac{33}{31} + o(1)}
\end{equation}
holds for almost all primes $p$. The exponent in the lower bound of \eqref{Neq1}
was further refined to $13/11$
by Garaev, Garcia and Konyagin \cite{GGK}, albeit, for infinitely many primes.
In this context, we prove the following theorem.
\begin{thm}\label{thm1}
Let $f$ be a non-CM normalized cuspidal Hecke eigenform of even weight $k \geq 2$ 
for $\Gamma_0(N)$ having integer Fourier coefficients $\{a_f(n) : n \in \N\}$ and 
let $\e>0$ be a real number.
Then we have
$$
P(a_f(p)) > (\log p)^{1/8}(\log\log p)^{3/8-\e}
$$
 for almost all primes $p$.
\end{thm}

We say that a subset $S$ of natural numbers has natural density one if 
$$
\lim_{x \to \infty} \frac{\#\{ n \leq x : n \in S\}}{x} 
$$
exists and is equal to $1$. 
As a corollary to \thmref{thm1}, we also have the following result
which improves the lower bound proved in \cite{MMS}.
\begin{cor}\label{cafn}
Let $f$ be as in \thmref{thm1} and let $\e > 0$ be a real number. Then the set
$$
\left\{n : a_f(n) = 0 \tx{ or } P(a_f(n)) >  (\log n)^{1/8} (\log\log n)^{3/8 -\e} \right\}
$$
has natural density equal to $1$.
\end{cor}
Now suppose that the Generalized Riemann Hypothesis (GRH), i.e., the 
Riemann Hypothesis for all Artin L-series is true. Then as pointed out in \cite{MMS}, 
from the work of Murty and Murty \cite{MMprime}, it follows that
\begin{equation}\label{lp}
P(a_f(p)) > e^{(\log p)^{1-\e}}
\end{equation}
for almost all primes $p$. Conditionally on GRH, we prove the following result.

\begin{thm}\label{thm2}
Suppose that GRH is true and let $f$ be as in \thmref{thm1}. 
For any real valued non-negative function $g$ satisfying the property 
$g(x) \rightarrow 0$ as $x \rightarrow \infty$, we have
$$
P(a_f(p)) > p^{g(p)}
$$
for almost all primes $p$.
\end{thm}

\begin{rmk}
We note that \thmref{thm2} gives an improvement of \eqref{lp} if we
choose $g(x) = 1/\log\log x$ for $x \ge 3$.
\end{rmk}
We can further improve the lower bound in \thmref{thm2} for a 
subset of primes with positive lower density. More precisely, we 
prove the following.

\begin{thm}\label{thm3}
Suppose that GRH is true and let $f$ be as in \thmref{thm1}. 
There exists a positive constant $c$ depending on $f$ such that the set of primes $p$ for
which  $a_f(p) \neq 0$ and 
$$
P(a_f(p)) >  c p^{1/14}(\log p)^{2/7}
$$
has lower density at least $1-\frac{2}{13(k-1)}$.
\end{thm}

As a corollary, we have the following result.
\begin{cor}\label{cafnG}
Suppose that GRH is true and let $f$ be as in \thmref{thm1}. For any  real valued 
non-negative function $g$ satisfying the property $g(x) \to 0$ as $x \to \infty$, the set
$$
\left\{n : a_f(n) =0 \tx{ or } P(a_f(n)) > n^{g(n)}\right\}
$$
has natural density equal to $1$.
\end{cor}
As in the case of prime numbers, the lower natural density of a subset $S$ of natural numbers
is defined by
$$
\liminf_{x \to \infty} \frac{\#\{ n \leq x : n \in S\}}{x}. 
$$
Here we have the following result.
\begin{cor}\label{cafnG2}
Suppose that GRH is true and let $f$ be as in \thmref{thm1}.  Then the lower natural density 
of the set
$$
\left\{n : a_f(n) =0 \tx{ or } P(a_f(n)) > n^{1/70}(\log n)^{1/7}\right\}
$$
is at least $1-e^{-c_1}$, where $c_1 = \(1-\frac{1}{6(k-1)}\) \log\(\frac{5}{4.9}\)$.
\end{cor}

In a recent work, Bennett, Gherga, Patel and Siksek \cite{BGPS}
proved that for any prime $p$ with $\tau(p) \neq 0$ and 
for any $n \geq 2$, 
\begin{equation}\label{taum2}
P(\tau(p^n)) ~\ge~ C\log\log p,
\end{equation}
where $C >0$ is a constant depending on $n$.
In this context, we have the following results.
\begin{cor}\label{thpn}
Let $f$ be as in \thmref{thm1}.
Let $\e>0$ be a real number. Then for almost all primes $p$ and for all $n \ge 1$, we have
$$
P(a_f(p^{2n+1})) > (\log p)^{1/8}(\log\log p)^{3/8-\e}.
$$
\end{cor}

\begin{cor}\label{thpn1}
Suppose that GRH is true and let $f$ be as in \thmref{thm1}.  
\begin{itemize}
\item
Then for a set of primes of lower density 
at least $1-\frac{2}{13(k-1)}$ and for all $n \ge 1$,
we have
$$
P(a_f(p^{2n+1})) >  c p^{1/14}(\log p)^{2/7}.
$$
 Here $c$ is a positive constant depending on $f$.
\item
Let $g$ be any non-negative real valued function satisfying the property 
$g(x) \rightarrow 0$ as $x \rightarrow \infty$
. Then for almost all primes $p$ and for all $n \ge 1$,
we have
$$
P(a_f(p^{2n+1})) > p^{g(p)}.
$$
\end{itemize}
\end{cor}

Next we generalize the result \eqref{taum2}
to normalized cuspidal Hecke eigenforms whose Fourier coefficients 
are not necessarily rational integers.  From now on, for any algebraic integer $\a$ in a number 
field $K$, we define $P(\a)$ to be the largest prime factor of the absolute norm
$\cN_{K}(\a)$ with the convention that $P(\a) =1$, if $\a = 0$ or a unit in $\cO_K$. 
With these conventions, we prove the following theorem.

\begin{thm}\label{thm6}
Let $f$ be a normalized cuspidal Hecke eigenform of even weight $k \ge 4$ 
and level $N$ with
Fourier coefficients $\{a_f(n) : n \in \N\}$. 
For any integer $n \geq 3$ and for any prime $p$ with $a_f(p^{n-1}) \neq 0$, we have
\begin{equation*}
P(a_f(p^{n-1})) ~\geq~ c(n_3, f)~ \log\log p~,
\end{equation*}
where $n_3 = \min\{ d \div n : d \geq 3\}$ and $c(n_3, f)$ is a constant depending only 
on $n_3$ and $f$.
\end{thm}

Also we derive a conditional lower bound for the largest prime factor of $a_f(p^{n-1})$
when $p$ is fixed and $n$ varies. For $p \nmid N$, let $\a_p, \b_p$ be the roots of the polynomial 
$x^2 - a_f(p)x + p^{k-1}$ and $\gamma_{p} = \a_p / \b_p$. For any
prime ideal $\fP$ of the ring of integers of $\Q(\gamma_{p})$, let $\nu_{\fP}$ denote the $\fP$-adic 
valuation. Also let $\varphi$ denote the Euler-phi function and $\omega(n)$
denote the number of distinct prime factors of $n$.

\begin{thm}\label{pafp}
Let $f$ be a normalized cuspidal Hecke eigenform of even weight $k \geq 2$ and level $N$ with
Fourier coefficients $\{a_f(n) : n \in \N\}$.  Also let $p$ be a fixed prime with $p \nmid N$. 
Suppose that $\gamma_{p}$ is not a root of unity and 
there exists a positive constant $r$ such that
\begin{equation*}
\nu_{\fP}\( \gamma_{p}^{\cN(\fP)-1}-1\) \leq r
\end{equation*}
for any prime ideal $\fP$ of the ring of integers of $\Q(\gamma_{p})$.
Then 
\begin{equation*}
P(a_f(p^{n-1})) ~>~ \frac{h(\gamma_p)}{26 rh_{f,p}} \cdot \frac{\varphi(n)^2}{(2d_fh_{f,p})^{\w(n)+1}}
\end{equation*}
for all sufficiently large $n$ depending on $f$ and $p$.
Further, when $f$ has rational integer Fourier coefficients and $p >3$, 
we have
\begin{equation*}
P(a_f(p^{n-1})) ~>~ \frac{(k-1-2\nu_{f,p})\log p}{52r} \cdot \frac{\varphi(n)^2}{2^{\omega(n)}}
\end{equation*}
for all sufficiently large $n$ (depending on $f$ and $p$).
Here $\nu_{f,p}$ is the $p$-adic valuation of $a_f(p)$, 
$d_f$ is the degree of $\K_f$ over $\Q$, $h_{f,p}$ is the class 
number of $\Q(\a_p)$ and $\cN$ is the absolute norm
on $\Q(\gamma_{p})$.
\end{thm}

\thmref{pafp} is a consequence of an effective number field analogue
of a result of Murty and S\'{e}guin (see \cite[Theorem 1.1]{MS}).

\begin{rmk}
Let the notations be as before. Also let $K$ be a number field, $\alpha \in K\setminus\{0\}$ 
be not a root of unity and let $\fp$ be a prime ideal of $\cO_K$ such that $\nu_{\fp}(\alpha) = 0$. 
We say that $\fp$ is Wieferich prime for $\a$ in $K$  
if $\nu_{\fp}\(\alpha^{\cN_K(\fp) -1} -1\) \geq 2$ and 
is called a super Wieferich prime for $\a$ in $K$ 
if $\nu_{\fp}\(\alpha^{\cN_K(\fp) -1} -1\) \geq 3$. Here $\cN_K$ denotes 
the absolute norm on $K$.  

Write $\a = \frac{\beta_1}{\beta_2}$, where
$\beta_1, \beta_2$ are in $\cO_K$ which are relatively prime to $\fp$.
It is easy to see that $(\beta_1^{\cN_K(\fp) -1} - \beta_2^{\cN_K(\fp) -1})\cO_K$ 
is divisible by $\fp$. We expect 
\begin{equation}\label{FQ}
\frac{(\beta_1^{\cN_K(\fp) -1} - \beta_2^{\cN_K(\fp) -1})\cO_K}{\fp}
\end{equation}
to be a random ideal in $\cO_K$. 
It follows from Tauberian theorem that the number of
ideals in $\cO_K$ with norm less than or equal to $x$ is asymptotic to
$\rho_K x$ and the number of ideals in $\cO_K$ which are divisible by $\fp$ 
and with norm less than or equal to $x$ is asymptotic to $\frac{\rho_K}{\cN_K(\fp)} x$
as $x$ tends to infinity.
Here $\rho_K$ is the residue of the Dedekind zeta function
$\zeta_K$ at $s=1$.  Thus the probability that the quantity in \eqref{FQ}
is divisible by $\fp$ is expected to be $\frac{1}{\cN_K(\fp)}$. Similarly,  the probability that
the quantity in \eqref{FQ} 
is divisible by $\fp^2$ is expected to be $\frac{1}{\cN_K(\fp)^2}$.
Since $\sum_{\fp}\frac{1}{\cN_K(\fp)^2}$ converges, it is expected that
there are only finitely many super Wieferich primes for $\alpha$ 
in $K$. Thus, we expect that $\nu_{\fp}\(\alpha^{\cN_K(\fp) -1} -1\) $ 
is bounded uniformly as we vary over prime ideals $\fp$ in $\cO_K$.
In particular, for any fixed  prime $p\nmid N $ such that 
$\gamma_{p}$ is not a root of unity, 
it is expected that $\nu_{\fP}\( \gamma_{p}^{\cN(\fP)-1}-1\)$ is 
bounded uniformly as we vary $\fP$ over prime ideals of the 
ring of integers of $\Q(\gamma_{p})$. 
When $K =\Q$, similar heuristics can be found in
\cite[Section 3]{CDP} and \cite{MS}. 

\end{rmk}

\medskip

\section{Preliminaries}

\subsection{Applications of $\ell$-adic Galois representations and 
Chebotarev density theorem}

Let $f$ be a normalized cuspidal Hecke eigenform of even weight $k \geq 2$
and level $N$ having rational integer Fourier coefficients $\{a_f(n): n \in \N \}$.
For any integer $d > 1$ and real number $x >0$, set
\begin{equation*}
\begin{split}
&\pi_f(x,d) = \#\{p \leq x : a_f(p) \equiv 0 ~(\mod d) \},\\
&\pi_f^*(x,d) = \#\{p \leq x :a_f(p) \neq 0~,~ a_f(p) \equiv 0 ~(\mod d) \}.
\end{split}
\end{equation*}
Let $\G  = \text{Gal}\(\overline{\Q}/\Q\)$ and for a prime $\ell$, let $\Z_\ell$ 
denote the ring of $\ell$-adic integers. By the work of Deligne \cite{De}, 
there exists a continuous representation
\begin{equation*}
\rho_{d,f} : {\G} \rightarrow {\GL}_2\(\prod_{\ell | d} \Z_\ell\)
\end{equation*}
which is unramified outside $dN$. Further, if $p \nmid dN$, then 
$$
\text{tr}\rho_{d,f}(\sigma_p) = a_f(p) \phantom{mm}\text{and}\phantom{mm} 
\text{det}\rho_{d,f}(\sigma_p) = p^{k-1},
$$
where $\sigma_p$ is a Frobenious element of $p$ in $\G$ and $\Z$ is embedded 
diagonally in $\prod_{\ell | d} \Z_\ell$. 
Denote by $\tilde{\rho}_{d,f}$ the reduction  of $\rho_{d,f}$ modulo $d$:
\begin{equation*}
\tilde{\rho}_{d,f} : {\G} \xrightarrow[]{\rho_{d,f}} {\GL}_2\(\prod_{\ell | d} \Z_\ell\)
\twoheadrightarrow {\GL}_2(\Z/ d\Z).
\end{equation*}
Suppose that $H_d$ is the kernel of $\tilde{\rho}_{d,f}$, $K_d$ is the subfield of 
$\overline{\Q}$ fixed by $H_d$ and 
${\G}_d = \text{Gal}(K_d/\Q) \cong \text{Im}\(\tilde{\rho}_{d,f}\)$. Further
suppose that $C_d$ is the subset of $\tilde{\rho}_{d,f}(\G)$ consisting of elements of trace 
zero. Let us set $\delta(d) = \frac{|C_d|}{|{\G}_d|}$.
For any prime $p\nmid dN$, the condition $a_f(p) \equiv 0 ~(\mod d)$ is equivalent to
 the fact that $\tilde{\rho}_{d,f}(\sigma_p) \in C_d$, where $\sigma_{p}$ is a Frobenious 
element of $p$ in $\G$. Hence 
by the Chebotarev density theorem applied to $K_d / \Q$, we have
\begin{equation*}
\pi_f(x,d) \sim \frac{|C_d|}{|{\G}_d|}\pi(x) = \delta(d)\pi(x)
\phantom{m}\text{as } x \to \infty.
\end{equation*}
From the works of Lagarias and Odlyzko \cite{LO} (see also \cite[Proof of Theorem 6.2]{MMS}), 
we can deduce the following lemma.
\begin{lem}\label{cheb}
Let $f$ be a non-CM normalized cuspidal Hecke eigenform of
even weight $k \geq 2$ and level $N$ with integer 
Fourier coefficients $\{a_f(n): n \in \N\}$. Then there exists a constant $c > 0$ depending on $f$ 
such that for any natural number $1< d \le (\log x)^c$, we have
$$
\pi_f(x,d) = \delta(d) \pi(x)+ {O} \(x~\exp\(\frac{- c_1\sqrt{\log x}}{d^2}\)\),
$$
where $c_1> 0$ is a constant which depends on $f$.
\end{lem}

The following lemma follows from the works of Ribet \cite{Rib} and 
Serre \cite{Sr} (see \cite{MMprime}).
\begin{lem}\label{Z0}
Suppose that $f$ is non-CM. Then for any $\e>0$, we have
$$
Z(x)= \# \{p \leq x : a_f(p)=0 \}  \ll_{\e} \frac{x}{(\log x)^{3/2-\e}}.
$$
Further, suppose that GRH is true. Then we have
$$
Z(x) \ll x^{3/4}.
$$
\end{lem}
Applying \lemref{Z0}, Murty and Murty  \cite{MMprime} deduced the following result.
\begin{thm}\label{chebGRH}
Suppose that GRH is true and $f$ is non-CM. 
Then for $x \geq 2$, we have
$$
\pi_f^*(x,d) = \delta(d) \pi(x) ~+~ O\(d^3 x^{1/2} \log(dNx)\) ~+~ O( x^{3/4} ).
$$
\end{thm}

From the recent result of Thorner and Zaman (\cite{TZ}, Theorem 1.1), we have the following theorem.
\begin{thm}\label{TZc}
Suppose that $f$ is non-CM. There exists an absolute 
constant $c>0$ such that for any natural number $d > 1$ 
and for $\log x > c d^4 \log dN$, we have
\begin{equation*}
\pi_f(x,d) ~\ll~ \d(d) \pi(x).
\end{equation*}
\end{thm}
We will also need the following lemma (see \cite[Proof of Theorem 3]{GM}, 
\cite[Lemma 5.4]{MMprime}, \cite[Section 4]{Sr}).
\begin{lem}\label{deltal}
For any prime $\ell$, we have
$$
\d(\ell) = \frac{1}{\ell} +O\(\frac{1}{\ell^2}\) 
\phantom{mm}\text{and}\phantom{mm}
 \d(\ell^n) =O\(\frac{1}{\ell^n}\)
$$
for any $n \in \N$.
\end{lem}

\subsection{The Sato-Tate conjecture}
\label{ssst}
One of our principal tools is the {\em Sato-Tate conjecture}, 
proved  by Barnet-Lamb, Geraghty, Harris and 
Taylor \cite[Theorem~B]{BLGHT11} (see 
also \cite{{CHT08},{HSBT10}}). 
It states that, for a non-CM normalized cuspidal Hecke eigenform~$f$, 
the numbers ${\lambda_f(p) = a_f(p)/p^{(k-1)/2}}$ are
equi-distributed in the interval ${[-2, 2]}$ with respect 
to the Sato-Tate measure 
${(1/\pi)\sqrt{1- t^2/4}~dt}.$ 
This means that for ${-2\le a\le b\le 2}$ the density of the set of primes~$p$ 
satisfying ${\lambda_f(p) \in [a, b]}$ is 
\begin{equation*}
\frac1\pi\int_a^b\sqrt{1-\frac{t^2}{4}} dt.
\end{equation*}

\smallskip
\subsection{Height of an algebraic number}
Let $K$ be a number field and $M_K$ be the set of places of $K$ normalized to 
extend the places of $\Q$. Also let $M_K^{\infty}$ be the subset of $M_K$ 
of infinite (Archimedean) places. For any prime $p$, 
let $\Q_p$ be the field of $p$-adic numbers and for each $v \in M_K$, let $K_v$ 
be the completion of $K$ with respect to the place $v$. For $v \in M_K$, let $d_v$ denotes 
the local degree at $v$ given by $d_v = [K_v : \Q_p]$ if $v \mid p$ 
and $d_v = [K_v : \R]$ if $v \in M_{K}^{\infty}$.
For $\alpha \in K\setminus \{0\}$,
the usual absolute logarithmic height of $\a$  is defined by
$$
h(\a) = \frac{1}{[K: \Q]} \sum_{v \in M_K} d_v \log^+ |\a|_v 
$$
where $\log^+ = \max\{\log , 0\}$.  It can be deduced from the definition of height that
\begin{equation}\label{hgt}
	h(\a) = \frac{1}{[K:\Q]} \( \sum_{\sigma : K \hookrightarrow \C} \log^{+}|\sigma(\a)| 
	~+~ \sum_{\fp} \max\{ 0, -\nu_{\fp}(\a)\} \log \cN_K(\fp) \),
\end{equation}
where the first sum runs over the embeddings of $K$ in $\C$ and the second 
sum runs over the prime ideals of $\cO_K$.
If $\alpha \in K$ is non-zero and not a root of unity, it is well known
that (see~\cite{ {BGH}, {ED}, {PV}}) 
\begin{equation}\label{lbh}
h\( \a \) ~ \geq~ \frac{1}{4d (\log^*d)^3}~,
\end{equation}
where $d = [\Q(\a): \Q]$ and $\log^* = \max \{ 1, \log  \}$.

\smallskip

\subsection{Prerequisites from cyclotomic polynomials and other required results}
Let $\Phi_n(x,y)$ be the $n$-th cyclotomic polynomial given by
$$
\Phi_n(x,y) = \prod_{\substack{m=1 \\ (m,n)=1}}^{n} \(x- e^{\frac{2\pi i m}{n}}y\).
$$
Let $K$ be a number field  with ring of integers $\cO_K$ and let $A, B \in \cO_K \setminus \{0\}$. 
A prime ideal $\fp$ of $\cO_K$ is called a primitive divisor of $A^n-B^n$ if 
$\fp \div A^n-B^n$ but $\fp \nmid A^m-B^m$ for any $1 \leq m < n$. 
The following two results follow from the work of Schinzel \cite{Sc}.
In the second Lemma, we have used the inequality \eqref{lbh}.
\begin{lem}\label{impr}
Let $K$ be a number field and $A, B \in  \cO_K \setminus \{0\}$ be such that $(A,B)=1$ 
and $A/B$ is not a root of unity. For any natural number $n > 2(2^m-1)$, $m=[\Q(A/B):\Q]$, 
if $\fp$ is a prime ideal of $\cO_K$ such that $\fp \div \Phi_n(A,B)$ and $\fp$ 
is not a primitive divisor of $A^n-B^n$, then
$$
\nu_{\fp}(\Phi_n(A,B)) ~ \leq~ \nu_{\fp} (n\cO_K).
$$
\end{lem}

\begin{lem}\label{Nphin}
Let $K$ be a number field and $A, B \in  \cO_K \setminus \{0\}$. 
If $(A,B)=1$ and $A/B$ is not a root of unity, then
\begin{equation}
\log|\cN_K(\Phi_n(A,B))| = h(A/B) [K:\Q] ~\varphi(n) \(1+O\(\frac{2^{\w(n)} \log(n+1)}{\varphi(n)}\)\),
\end{equation}
where the implied constant depends only on the degree of $\Q(A/B)$ over $\Q$. 
Here $h(\cdot)$ denotes the absolute logarithmic height of an algebraic number
and $\cN_K$ denotes the absolute norm on $K$.
\end{lem}

We need the following Brun-Titchmarsh inequality (see \cite{HR}, Theorem 3.8).
\begin{thm}
If $1 \leq n< x$ and $ (a, n)=1$, then
\begin{equation}\label{BT}
\pi(x;n,a) = \#\left\{p \leq x  : p \equiv a (\mod n) \right\} ~<~  \frac{3x}{\varphi(n) \log(x/n)}.
\end{equation}
\end{thm}

 We also need the following theorem due to Bugeaud \cite{YB}.
\begin{thm}\label{Bg}
Let $m, n \geq 2$ be rational integers with $mn \geq 6$. Also let $K$ be an algebraic 
number field and $\a, \b$ be non-zero algebraic integers of $K$ 
with absolute height at most $A ~(\geq 3)$.
Suppose that $x, y$ are coprime algebraic integers in $K$ and let $\e > 0$ be a real number. 
Then there exists an effectively computable constant $c$ 
depending on $\e, m, n, A$ and $K$ such that
$$
P(\a x^m + \b y^n) ~\geq~ \frac{\log\log \max\{|\cN_K(x)|, |\cN_K(y)|\}}{(7+\e)~ mn[K:\Q]  ~\min\{m, n\}}
$$
provided that $\max\{|\cN_K(x)|, |\cN_K(y)|\} \geq c$. Here $\cN_K$ denotes the absolute norm on $K$.
\end{thm}

\smallskip

\section{Largest prime factor of Fourier coefficients at primes}
In this section, we will derive lower bounds for the largest prime factor of 
Fourier coefficients of a non-CM normalized cuspidal Hecke eigenform.
Now we shall prove the following lemma
which is essential for the proofs of the theorems.

\begin{lem}\label{lemlogfA}
Let $f$ be a non-CM normalized cuspidal Hecke eigenform 
of even weight $k \geq 2$ for $\Gamma_0(N)$ 
having integer Fourier coefficients $\{a_f(n) : n \in \N\}$. 
Let $A \subseteq \{p : a_f(p) \neq 0\}$ be a set of primes 
having positive upper density. 
Then there exists a strictly increasing sequence $\(x_n\)_{n \geq 1}$ of natural numbers such that 
$$
\sum_{\substack{p \leq x_n \\ p \in A}} \log|a_f(p)| ~\gg~ x_n
$$ 
for $n \in \N$. Here the implied constant depends on $f$ and $A$.
\end{lem}
\begin{proof}
Let $A(x) = \{ p \leq x : p \in A\}$ and 
\begin{equation}
	\eta ~=~ \limsup_{x \rightarrow \infty} \frac{\#A(x)}{\pi(x)} ~>~ 0.
\end{equation}
This implies that there exists a strictly increasing sequence 
of natural numbers 
$(x_n)_{n \geq 1}$ such that
\begin{equation}\label{lowdenS}
	\frac{\#A(x_n)}{\pi(x_n)} ~\geq~ \frac{\eta}{2}~,
	\phantom{m} \forall ~~ n \in \N.
\end{equation}
	From the Sato-Tate conjecture (proved in~\cite{BLGHT11}, 
	see Subsection~\ref{ssst}) it follows that there exists a 
	positive real number $M=M(\eta)$ such that
\begin{equation}\label{lowdenT}
	\#\left\{p \leq x :  |a_f(p)| \geq \frac{p^{(k-1)/2}}{M}  \right\} > \(1-\frac{\eta}{4}\) \pi(x)
\end{equation}
for all sufficiently large $x$.
Let us set $T=\{p :  |a_f(p)| \geq p^{(k-1)/2}/M\} $ 
and  $W = A \cap T$ . Also let
$T(x) = T \cap \{p: p\leq x\}$ and $W(x) = W \cap \{p: p\leq x\}$. 
Then from \eqref{lowdenS} and \eqref{lowdenT},
we have
$$
\#W(x_n)~=~ \#A(x_n) + \#T(x_n) - \#(A \cup T)(x_n) 
~\geq~
\frac{\eta}{4} \pi(x_n)
$$
for all sufficiently large $n$. This implies that
\begin{eqnarray*}
	\sum_{p \in A(x_n)} \log|a_f(p)| 
	~\geq~ 
	\sum_{p \in W(x_n)} \log|a_f(p)| 
	& \geq &
	\sum_{\substack{p \in W(x_n)\\ p > M^2}} \log\(\frac{p^{(k-1)/2}}{M}\)\\
	& \geq &
	\frac{k-1}{2} \sum_{\substack{p \in W(x_n)\\ p > M^2}} \log p ~+~ O\(\# W(x_n) \log M \).
\end{eqnarray*}
By partial summation formula, we deduce that
$$
\sum_{\substack{p \in W(x_n)\\ p > M^2}} \log p \gg x_n
$$
for all sufficiently large $n$. Thus we obtain
$$
\sum_{p \in A(x_n)} \log|a_f(p)| \gg x_n
$$
for all sufficiently large $n$. This completes the proof.
\end{proof}

\subsection{Proof of \thmref{thm1}}
Let $f$ be a non-CM normalized cuspidal Hecke eigenform of even weight $k \geq 2$
for $\Gamma_0(N)$ 
having integer Fourier coefficients $\{a_f(n) : n \in \N\}$.
Let $ \e > 0$ be a real number and consider the set
\begin{equation*}
S = \left\{p  ~:~ a_f(p) \neq 0, ~ P(a_f(p)) \leq (\log p)^{1/8}(\log\log p)^{3/8-\e}\right\}.
\end{equation*}
Also let $S(x) = S \cap \{p: p\leq x\}$. We want to show that
\begin{equation}\label{SD0}
	\lim_{x \rightarrow \infty} \frac{\#S(x)}{\pi(x)} =0.
\end{equation}
Suppose not. Then
\begin{equation}\label{lowdenS>0}
\eta = \limsup_{x \rightarrow \infty} \frac{\#S(x)}{\pi(x)} > 0.
\end{equation}

Let $\cP_{S}(x) = \left\{q ~:~   q \text{ divides }  \(\prod_{p \in S(x)} a_f(p) \), ~q ~\text{prime}  \right\}$ 
and write
\begin{equation}\label{prodafp}
\prod_{p \in S(x)} |a_f(p)| = \prod_{q \in \cP_{S}(x)} q^{\nu_{x, q}},
\end{equation}
where
$$
\nu_{x, q} = \sum_{p \in S(x)} \nu_q(a_f(p)).
$$
Using Deligne's bound, we have
\begin{eqnarray}\label{vxq}
 \nu_{x,q} 
 ~\leq~ 
 \sum_{\substack{p \leq x \\ a_f(p) \neq 0}} \nu_q(a_f(p))\
 ~=~
 \sum_{\substack{p \leq x \\ a_f(p) \neq 0}} \sum_{\substack{m \geq 1 \\ q^m | a_f(p)}} 1
&=&
 \sum_{1 \leq m \leq \frac{\log(2 x^{(k-1)/2})}{\log q}} 
\sum_{\substack{p \leq x,  ~a_f(p) \neq 0 \\ a_f(p) \equiv 0 (\mod q^m)}} 1  \nonumber \\
 &=& 
 \sum_{1 \leq m \leq \frac{\log(2 x^{(k-1)/2})}{\log q}}  \pi_{f}^{*}(x, q^m).
 \end{eqnarray}
 Applying  \thmref{TZc}, there exists a positive constant $c$ 
 depending on $f$ such that for $1< d \leq c \frac{(\log x)^{1/4}}{(\log\log x)^{1/4}}$, we have
 \begin{equation}\label{unpixd}
\pi_f^*(x, d) \ll \d(d) \pi(x).
 \end{equation}
Set
\begin{equation}
 z=c \frac{(\log x)^{1/4}}{(\log\log x)^{1/4}} \phantom{mm} \text{and}\phantom{mm}
 y=  (\log x)^{1/8}(\log\log x)^{3/8-\e}.
\end{equation}
For $q \in \cP_S(x)$, let  $m_0 = m_0(x, q)=  \big[\frac{\log z}{\log q}\big]$.  Note that
$m_0 \geq 1$ for all $x$ sufficiently large. We estimate the sum
in \eqref{vxq} by dividing it into two parts; for $1\le m \le m_0$ and then 
for $m> m_0$. It follows from \eqref{unpixd} that the first sum is
 \begin{eqnarray}\label{vxq1}
\sum_{1 \leq m \leq m_0}  \pi_{f}^{*}(x, q^m) 
&\ll& 
\sum_{1 \leq m \leq m_0} \d(q^m) \pi(x) ~\ll~
\sum_{1\leq m \leq m_0} \frac{\pi(x)}{q^m}
~\ll~ 
\frac{\pi(x)}{q} .
 \end{eqnarray}
 The inequality in \eqref{vxq1} follows from \lemref{deltal} and by noticing that $q \le y$. 
 The second sum is equal to
 \begin{eqnarray}\label{vxq2}
\sum_{m_0 < m \leq \frac{\log(2 x^{(k-1)/2})}{\log q}} \pi_{f}^{*}(x, q^m)
 &\leq&
 \pi_{f}^{*}(x, q^{m_0}) \sum_{m_0 < m \leq \frac{\log(2 x^{(k-1)/2})}{\log q}} 1 \nonumber\\
 &\ll&
 \d(q^{m_0}) \pi(x)  \cdot \frac{\log x}{\log q}
\nonumber\\
 & \ll &
\frac{\pi(x)\log x}{q^{m_0} \log q}
 ~\ll~ 
 \frac{x}{z} \cdot~\frac{q}{\log q}.
\end{eqnarray}
From \eqref{vxq1} and \eqref{vxq2}, we deduce that
\begin{equation}\label{vxq3}
\nu_{x,q} ~\ll~ \frac{x}{z} \cdot \frac{q}{\log q}.
\end{equation}
It follows from \eqref{prodafp} that
\begin{equation}
	\sum_{p \in S(x)} \log|a_f(p)| ~= \sum_{q \in \cP_S(x)} \nu_{x, q} \log q.
\end{equation}
Thus applying \eqref{vxq3}, we obtain
\begin{equation}\label{upp}
	\begin{split}
	 \sum_{p \in S(x)} \log|a_f(p)|
	 ~=~
		\sum_{q \in \cP_S(x)}\nu_{x,q} \log q 
		~\ll~  
		\frac{x}{z}\sum_{q \leq y}q  
		~\ll~ 
		\frac{x}{z}. \frac{y^2}{\log y} 
		~\ll~ 
		\frac{x}{(\log\log x)^{2\e}}
	\end{split}
\end{equation}
for all sufficiently large $x$.
From \eqref{lowdenS>0} and \lemref{lemlogfA}, we deduce that there exists a strictly increasing sequence $(x_n)_{n\geq 1}$ of natural numbers such that
$$
\sum_{p \in S(x_n)} \log|a_f(p)| \gg x_n .
$$
This is a contradiction to \eqref{upp} 
for sufficiently large $n$. Now \thmref{thm1} follows from \eqref{SD0} and \lemref{Z0}.

\begin{rmk}
 We note that for any real valued non-negative function $u$ 
 satisfying $u(x) \to 0$ as $x \to \infty$, the lower bound 
 $(\log p)^{1/8}(\log\log p)^{3/8-\e}$ in \thmref{thm1} can be replaced 
 with $(\log p)^{1/8}(\log\log p)^{3/8}u(p)$.
\end{rmk}

\medskip

\subsection{Proof of \thmref{thm3}}
The proof follows along the lines of proof of \thmref{thm1}.
Let 
\begin{equation*}
S = \left\{p  : a_f(p) \neq 0, ~ P(a_f(p)) \leq c p^{1/14}(\log p)^{2/7}\right\},
\end{equation*}
where $0 < c <1$ is a constant which will be chosen later.
We will show that
\begin{equation}\label{limsup}
\eta 
=
\limsup_{x \rightarrow \infty} \frac{\#S(x)}{\pi(x)} 
~\leq~ 
\frac{2}{13(k-1)}.
\end{equation}
Suppose that \eqref{limsup} is not true i.e., $\eta > \frac{2}{13(k-1)}$. Choose 
$\e>0$ such that 
\begin{equation}\label{cep}
\eta(1-2\e) > \frac{2}{13(k-1)}. 
\end{equation}
Then there exists a strictly increasing 
sequence of natural numbers $(x_n)_{n \geq 1}$ such that
$$
\frac{\#S(x_n)}{\pi(x_n)} \geq (1-\e) \eta, \phantom{m} \forall~~ n \in \N.
$$
Let
$$
\cP_{S}(x) = \left\{q ~:~   q \text{ divides }  \(\prod_{p \in S(x)} a_f(p) \), ~q ~\text{prime}  \right\}.
$$
Write
\begin{equation}\label{c1}
\prod_{p \in S(x)} |a_f(p)| = \prod_{q \in \cP_S(x)} q^{\nu_{x,q}},
\end{equation}
which implies that
\begin{equation}
\sum_{p \in S(x)} \log|a_f(p)| = \sum_{q \in \cP_S(x)} \nu_{x,q} \log q.
\end{equation}
As before, we have
\begin{equation}
	\nu_{x,q} ~\leq~ \sum_{1 \leq m \leq \frac{\log(2 x^{(k-1)/2})}{\log q}}  \pi_{f}^{*}(x, q^m).
\end{equation}
Set 
\begin{equation*}
z=c\frac{x^{1/7}}{(\log x)^{3/7}} \phantom{mm}\text{and}\phantom{mm}
y= c x^{1/14}(\log x)^{2/7}.
\end{equation*}
Also let  $m_0 = \Big[\frac{\log z}{\log q}\Big]$.
 Then applying \thmref{chebGRH} and \lemref{deltal}, we have 
\begin{equation}\label{pif1}
	\begin{split}
		\sum_{1 \leq m \leq m_0}  \pi_{f}^{*}(x, q^m) 
		&= \sum_{1 \leq m \leq m_0} \bigg\{\d(q^m)\pi(x)
		~+~ O\(q^{3m} x^{1/2} \log x \) ~+~ O\(x^{3/4}\)\bigg\}\\
		& = \frac{\pi(x)}{q} ~+~ O\(\frac{\pi(x)}{q^2}\) ~+~ O\(z^3 x^{1/2} \log x\)
		~+~ O\(x^{3/4} \frac{\log z}{\log q}\).
	\end{split}
\end{equation}
Again applying \thmref{chebGRH} and \lemref{deltal}, we get
\begin{equation}\label{pif2}
\begin{split}
	\sum_{m_0 < m \leq \frac{\log(2 x^{(k-1)/2})}{\log q}}  \pi_{f}^{*}(x, q^m) 
	&\leq~ 
	 \pi_{f}^{*}(x, q^{m_0}) \sum_{m_0 \leq m \leq \frac{\log(2 x^{(k-1)/2})}{\log q}} 1\\
	& \ll~ 
	 \(\frac{\pi(x)}{q^{m_0}} ~+~ q^{3m_0} x^{1/2} \log x ~+~ x^{3/4}\)  \frac{\log x}{\log q}\\
	 &\ll~
	  \frac{x}{z}\frac{q}{\log q} ~+~ z^3 x^{1/2}\frac{(\log x)^2}{\log q}
	  ~+~ x^{3/4}  \frac{\log x}{\log q}.
\end{split}
\end{equation}
From \eqref{pif1} and \eqref{pif2}, we get
\begin{equation}\label{pif3}
	\nu_{x,q} ~\leq~
	 \frac{\pi(x)}{q}  ~+~ O\(\frac{\pi(x)}{q^2}  ~+~  \frac{x}{z}\frac{q}{\log q}
	~+~  z^3 x^{1/2}\frac{(\log x)^2}{\log q}   ~+~ x^{3/4}  \frac{\log x}{\log q} \).
\end{equation}
Since $q \in \cP_S(x)$, we have $q \le y$ and so it follows from \eqref{pif3} that
\begin{equation*}
	\begin{split}
	\sum_{q \in \cP_S(x)} \nu_{x,q} \log q 
	&\leq  
	\pi(x) \log y + c_1\(\pi(x) + \frac{x}{z}\frac{y^2}{\log y} + z^3 x^{1/2} (\log x)^2 \frac{y}{\log y}
	~+~ x^{3/4} \log x  \frac{y}{\log y} \),
	\end{split}
\end{equation*}
where $c_1$ is a positive constant depending on $f$. Now we choose $c$ such that
$3000 c_1 c(1 + c^3) < 1$.
Then by substituting the values of $y$ and $z$, we obtain
\begin{equation}\label{uppRH}
		\sum_{q \in \cP_S(x)} \nu_{x,q} \log q < \frac{x}{13} 
\end{equation}
for all $x$  sufficiently large. On the other hand, from the Sato-Tate conjecture, there exists a 
positive real number $M=M(\eta, \e)$ such that
\begin{equation}
	\#\left\{p \leq x : |a_f(p)| ~\geq~ \frac{p^{(k-1)/2}}{M} \right\} > \(1-\frac{\e \eta}{2}\) \pi(x)
\end{equation}
for all sufficiently large $x$. As in the proof of \lemref{lemlogfA}, by using partial 
summation formula, we deduce that
\begin{equation}\label{lowRH}
\sum_{p \in S(x_n)} \log|a_f(p)| ~>~ \eta(1-2\e) \frac{k-1}{2} x_n
\end{equation}
for all sufficiently large $n$.
From \eqref{uppRH} and \eqref{lowRH}, we deduce that
$$
\eta(1-2\e) \frac{k-1}{2} ~\leq~ \frac{1}{13},
$$
a contradiction to \eqref{cep}. Hence we conclude that 
$$
\eta =\limsup_{x \rightarrow \infty} \frac{\#S(x)}{\pi(x)} \leq \frac{2}{13(k-1)}. 
$$
This completes the proof of \thmref{thm3}.

\medskip

\subsection{Proof of \thmref{thm2}}
Without loss of generality, we can assume that $g(x)\log x \rightarrow \infty$ as $x \rightarrow \infty$, 
 by replacing $g(x)$ with $g(x)+\frac{1}{\log\log x}$ for $x \ge 3$ if necessary.

Let $S_g= \{p : a_f(p) \neq 0, ~P(a_f(p)) \leq p^{g(p)}\}$ and 
$S_g(x) = S_g \cap \{p : p \leq x\}$. We will show that
\begin{equation}\label{lim0}
\lim_{x \rightarrow \infty} \frac{\# S_g(x)}{\pi(x)} ~=~ 0.
\end{equation}
Suppose that \eqref{lim0} is not true. Then there exists  an $\eta > 0$ and 
a strictly increasing sequence of natural numbers $(x_n)_{n \geq 1}$ such that
$$
\frac{\#S_g(x_n)}{\pi(x_n)} ~\geq~ \eta,
\phantom{m} \forall ~~ n \in \N.
$$
As before, let 
$$
\cP_g(x) = \left\{ q :  q \text{ divides }  \(\prod_{p \in S_g(x)} a_f(p) \), ~q ~\text{prime} \right\}.
$$ 
Write
\begin{equation}
	\prod_{p \in S_g(x)} |a_f(p)| = \prod_{q \in \cP_g(x)} q^{\nu_{g,x,q}},
\end{equation}
and hence
\begin{equation}
\sum_{p \in S_g(x)} \log|a_f(p)| = \sum_{q \in \cP_g(x)} \nu_{g,x,q} \log q.
\end{equation}
Let $y= x^{g(x)},~ z=y^2$. As before, for all sufficiently large $x$, we have
\begin{equation}\label{uvgx}
\sum_{q \in \cP_g(x)} \nu_{g,x,q} \log q \leq  \pi(x) \log y + O\(\pi(x) + \frac{x}{\log y}\)
 \ll x  g(x) + \frac{x}{\log x}+ \frac{x}{g(x)\log x}
\end{equation}
and
\begin{equation}\label{lvgx}
\sum_{p \in S_g(x_n)} \log|a_f(p)| \gg x_n
\end{equation}
for all sufficiently large $n$.
From \eqref{uvgx} and \eqref{lvgx}, we get a contradiction 
since $g(x_n) \rightarrow 0$ and $g(x_n)\log(x_n) \rightarrow \infty$ as $n \rightarrow \infty$. Thus we conclude 
that $P(a_f(p)) > p^{g(p)}$ for almost all primes.

\medskip

\section{Largest prime factor of $a_f(n)$}

\subsection{Proof of \corref{cafn}}
Let $\e>0$ be a real number and 
$$
\cP =\left\{p ~:~ a_f(p)\neq 0, ~P(a_f(p)) > (\log p)^{1/8} (\log\log p)^{3/8 -\e} \right\}.
$$
By \thmref{thm1}, $\cP$ has density equal to $1$. Let $x$ be sufficiently large and
$ x_0 =  x^{1/(\log\log x)^{\e}}$. We first show that 
$$
T = \left\{n ~:~ \exists ~p \in \cP \text{ with } p ~|~ n, ~ p > n^{1/(\log\log n)^{\e}} \right\}
$$
has natural density equal to $1$.  Set
$$
\sA = \{n : n \leq x\}, \phantom{m} \sP_0 =\{p : p \in \cP,~ p> x_0 \}, 
\phantom{m}z=x^{1/5} \phantom{m} \tx{and} \phantom{m}
\sP_0(z) =\{ p \in \cP_0 :  p < z \}.
$$
Applying Brun's Sieve \cite[Theorem 2.1]{HR} (with $b=1, \lambda=0.27, \kappa=1$), we get
$$
\sS(\sA, \sP_0, z)
= \#\{n \in \sA  : (n, ~ p )=1~ ~\forall p \in \sP_0(z) \} 
~\ll~
x \prod_{ \substack{x_0 < p < z \\ p \in \cP }} \(1-\frac{1}{p}\) 
~\ll~
 \frac{x}{(\log\log x)^{\eta}}
$$
for some $\eta >0$. Since $\{ \sqrt{x} \leq n \leq x : n \not\in T \} \leq \sS(\sA, \sP_0, z)$, we have
$$
\#\{n \leq x ~:~ n \in T  \} \sim x  \phantom{m}\text{as } x \to \infty.
$$ 
We also have
$$
\# \left\{n \leq x : \exists ~p \in \cP \tx{ with } p^2 |n,~ p >  n^{1/(\log\log n)^{\e}} \right\} 
~\leq~
 \sqrt{x}+\sum_{\sqrt{x_0} < p } 
\frac{x}{p^2} =o(x),
$$
by separating the terms for $n \leq \sqrt{x}$ and for $n > \sqrt{x}$.
Hence we conclude that the set 
$$
S =\left\{ n ~:~  \exists ~p \in \cP \tx{ with } p \mid\mid n,~ p > n^{1/(\log\log n)^{\e}} \right\}
$$
has natural density equal to $1$. Here $p \mid\mid n$ means $p \mid n$ but $p^2 \nmid n$.

Now for $n \in S$, if $a_f(n) \neq 0,$ then
\begin{equation}\label{Pnp}
P(a_f(n)) \geq P(a_f(p)) > (\log p)^{1/8} (\log\log p)^{3/8 -\e} > (\log n)^{1/8} (\log\log n)^{3/8 -2\e}
\end{equation}
for all sufficiently large $n \in S$. This completes the proof of the corollary.

\subsection{Proof of \corref{cafnG} }
Let
$$
\cP = \left\{ p ~:~ a_f(p) \neq 0, ~P(a_f(p)) > p^{1/14} (\log p)^{3/14}\right\}.
$$
By \thmref{thm3}, $\cP$ has lower density at least $1-\frac{2}{13(k-1)}$. 
Without loss of generality, we assume that $g(x)\log x \to \infty$ as $x \to \infty$.
By choosing $x_0 = x^{g(x)}$ and applying Brun's sieve as before, 
we can show that the set
$$
S =\left\{ n : \exists ~p \in \cP \tx{ with } p \mid\mid n,~ p > n^{g(n)} \right\}
$$
has natural density equal to $1$. Hence \corref{cafnG} follows by arguing as in \eqref{Pnp}.

\subsection{Proof of \corref{cafnG2}}
As before, let
$$
\cP = \left\{ p : ~a_f(p) \neq 0, ~P(a_f(p)) > p^{1/14} (\log p)^{3/14}\right\}.
$$
By \thmref{thm3}, $\cP$ has lower density at least $1-\frac{2}{13(k-1)}$. Let $x_0 = x^{1/5}$ and 
$$
 R = \left\{n ~:~ (n,p)=1\ \forall p \in \cP \tx{ with } p > n^{1/5} \right\}.
$$
Also let
$$
\sA=\{n : n\leq x\}, 
\phantom{m}
\sP_0=\{p : p \in \cP , ~p > x^{1/5}\} 
\phantom{m}\tx{and} \phantom{m} 
z= x^{1/4.9}.
$$
Applying Brun's sieve, we get
\begin{eqnarray*}
\#\{n \leq x ~:~ n \in R \} 
~\leq~
 \sS(\sA, \sP_0, z) 
&\leq& 
(1+o(1)) x \prod_{\substack{x_0 < p < z \\ p \in \cP}} \(1-\frac{1}{p}\)  \\
&=& 
(1+o(1)) x  ~\cdot \exp \(-\sum_{\substack{x^{\frac{1}{5}} <p < x^{\frac{1}{4.9}} \\ p \in \cP}} \frac{1}{p} \).
\end{eqnarray*}
By using partial summation, for all sufficiently large $x$, we have
$$
\sum_{\substack{x^{\frac{1}{5}} <p < x^{\frac{1}{4.9}} \\ p \in \cP}} \frac{1}{p} 
~\geq~ 
c_0 = \(1-\frac{2}{12.5(k-1)}\) \log\(\frac{5}{4.9}\).
$$
Thus $\#\{n \leq x ~:~ n \in R \}  \leq (e^{-c_0}+o(1)) x$ for all sufficiently large $x$. 
Hence we deduce that the lower natural density of the set
\begin{equation}\label{W5}
S = \{n ~ :~  \exists ~p \in \cP \tx{ with } p \mid\mid n,~ p > n^{1/5}\}
\end{equation}
is at least $1-e^{-c_1}$, where
$$
c_1 = \(1-\frac{1}{6(k-1)}\) \log\(\frac{5}{4.9}\).
$$
Arguing as before, \corref{cafnG2} follows along the lines of \eqref{Pnp}.

\begin{rmk}
The exponent $1/70$ of $n$ in \corref{cafnG2} is not the best possible.
 One can improve it to $\frac{1}{14\a}-\e $, where $\a = 2+\frac{2.01}{e^{0.54}-1} = 4.80723566 \cdots$ 
( and $\frac{1}{14\a} = 0.01485855 \cdots)$. 
 Indeed by choosing $x_0=x^{\frac{1}{\a+2\e}}, ~z=x^{\frac{1}{\a+\e}}$, one can show that
$$
S = \{n ~ :~  \exists ~p \in \cP \tx{ with } p \mid\mid n,~ p > n^{\frac{1}{\a+2\e}}\}
$$
has lower natural density at least $1-e^{-c_2}$, where
$$
 c_2 = \(1-\frac{2}{(13-\e)(k-1)}\) \log\(\frac{\a+2\e}{\a+\e}\).
$$
\end{rmk}

\section{Largest prime factor of Fourier coefficients at prime powers}

\subsection{Proofs of \corref{thpn} and \corref{thpn1} }
Let $f$ be a non-CM normalized cuspidal Hecke eigenform of  
even weight $k \geq 2$ for $\Gamma_0(N)$ 
having integer Fourier coefficients $\{a_f(n) : n \in \N\}$. Let $p$ be a rational prime
such that $p \nmid N$ and $a_f(p) \ne 0$. Then $a_f(p) \mid a_f(p^{2n+1})$ for any $n \geq 1$
follows from the recurrence formula
$$
a_f({p}^{n+1}) = a_f(p) a_f(p^{n}) - p^{k-1} a_f(p^{n-1}).
$$ 
From \lemref{Z0} and the fact that for all $p$ sufficiently large, either
$a_f(p) = 0$ or $a_f(p^n) \ne 0$ for all $n \ge 1$ 
 (see \cite[Lemma 3.2]{BGPS}, \cite[Lemma 2.2]{KRW}, \cite[Lemma 2.5]{MModd}),
 we see that the set 
 $$
 \{ p \leq x : p \nmid N, ~a_f(p^n) \ne 0~ \forall ~n \geq 1\}
 $$ 
 has density one.  
 Now \corref{thpn} and \corref{thpn1} follow by 
 applying \thmref{thm1}, \thmref{thm2} and \thmref{thm3}.

\subsection{Proof of \thmref{thm6}}
Let $f$ be a normalized cuspidal Hecke eigenform of even weight $k \ge 4$, level~$N$
with Fourier coefficients $\{a_f(n): n \in \N \}$. Also let
$\K_f = \Q(\{a_f(n): n \in \N \})$ and $d_f = [\K_f : \Q]$. 
Then for any rational prime $p \nmid N$ and any natural number $n \ge 1$, it is well known that
\begin{equation}\label{multi}
a_f(p^{n+1}) = a_f(p) a_f(p^{n}) - p^{k-1} a_f(p^{n-1}).
\end{equation}
Thus for $n \ge 1$, inductively, we have
\begin{equation}\label{lucp}
a_f(p^{n-1}) = \frac{\a_p^n - \b_p^n}{\a_p-\b_p},
\end{equation}
where $\a_p, \b_p$ are the roots of the polynomial $x^2 - a_f(p)x + p^{k-1}$.
For any natural number $m\ge 3$, let $\zeta_m = e^{\frac{2\pi i}{m}}$ and the polynomials 
$\Phi_m(X,Y),\Psi_m(X,Y) \in \Z[X,Y]$ be given by
\begin{equation*}
\Phi_m(X,Y) = \prod_{\substack{1 \leq j \leq m \\ (j, m)=1}} \(X- \zeta_m^j Y\)~,~
\Psi_m(X,Y) = \prod_{\substack{1 \leq j < m/2 \\ (j, m)=1}} \(X- 4 \cos^2\(\frac{\pi j}{m }\)Y\).
\end{equation*}
Then we have $\Phi_m(X,Y) = \Psi_m((X+Y)^2, XY)$. From \eqref{lucp}, we get
\begin{equation*}
a_f(p^{n-1}) = \prod_{\substack{d \div n \\d>1}} \Phi_d(\a_p, \b_p).
\end{equation*}
From now onwards assume that $n \geq 3$ and let $p$ be a prime 
such that $a_f(p^{n-1}) \neq 0$. 

If $a_f(p)=0$, then applying \eqref{multi}, we see that $n$ must be odd and
$P(a_f(p^{n-1}))=p$. Recall that for any $\a \in \cO_{\K_f} \setminus \{ 0 \}$ which is not a unit, 
$P(\a)$ is the largest prime factor of the absolute norm
$\cN_{\K_f}(\a)$.

Now assume that $a_f(p)\neq 0$.
If $a_f(p)$ and $p$ are not coprime, then there exists a prime ideal 
$\fp$ in $\cO_{\K_f}$ lying above the rational prime $p$ such that 
$\fp \div a_f(p)\cO_{\K_f}$. It then follows from \eqref{multi} that 
$\fp \div a_f(p^{n-1})\cO_{\K_f}$. Thus we have
$$
P(a_f(p^{n-1})) \geq p~.
$$
Now suppose that $a_f(p)$ and $p$ are coprime algebraic integers in $\cO_{\K_f}$.
Set
$$
n_3 = \min\{d \div n : d \geq 3 \}.
$$ 
Then we have
$$
P(a_f(p^{n -1})) ~\geq~  P(\Phi_{n_3}(\a_p, \b_p)) ~=~ P(\Psi_{n_3}(a_f(p)^2, p^{k-1})).
$$
Let  $\lambda_{n_3} = \z_{n_3} + \z_{n_3}^{-1} + 2 = 4\cos^2\(\frac{\pi}{n_3}\)$ and 
$L= \K_f(\lambda_{n_3})$.
We have
$$
P(\Psi_n(a_f(p)^2, p^{k-1})) ~\geq~ P(a_f(p)^2- \lambda_{n_3} p^{k-1}).
$$
Now by applying \thmref{Bg}, there exists a positive constant $c_1(n_3,f)$ 
depending on $n_3$ and $f$ such that
$$
P(a_f(p)^2- \lambda_{n_3} p^{k-1}) ~\geq~ \frac{\log\log p}{16(k-1)d_f \varphi(n_3)}
$$
provided $p > c_1(n_3, f)$.
Thus for all such primes $p$, we have
$$
P(a_f(p)^2- \lambda_{n_3} p^{k-1}) ~\geq~ c(n_3, f) \log\log p
$$
where $c(n_3, f)$ is a positive constant depending on $n_3$ and $f$.
Thus for $n \geq 3$, we have
$$
P(a_f(p^{n-1})) ~\geq~  c(n_3, f) \log\log p~.
$$

\subsection{Largest prime factor of $\Phi_n(A, B)$}

In \cite{St},  Stewart proved that if $\a, \b$ are complex numbers 
such that $(\a+\b)^2$, $\a\b$ are non-zero rational integers and $\a/\b$ is not a root of unity, then
\begin{equation}\label{Phi104}
P(\Phi_n(\a,\b)) > n ~\tx{exp} \(\frac{\log n}{104\log\log n}\)
\end{equation}
for $n> n_0$, where $n_0$ is a positive constant effectively computable in terms of $\w(\a\b)$ 
and the discriminant of the field $\Q(\a/\b)$. The constant $n_0$ was made explicit 
and the dependency of $n_0$ was refined by the first and the second author
along with Hong \cite{BGH} (see also \cite{HH}).

Let $f$ be a normalized cuspidal Hecke eigenform of even weight $k \geq 2$, level $N$ 
having Fourier coefficients $\{a_f(n) : n \in \N\}$. For $p \nmid N$,
let $\a_p, \b_p$ denote the roots of the polynomial $x^2-a_f(p)x+p^{k-1}$ 
and $\gamma_{p} = \a_p / \b_p$. If the Fourier coefficients of $f$ are rational integers, 
then from \eqref{Phi104}, it follows that for any rational prime $p \nmid N$ 
for which $\gamma_{p}$ is not a root of unity, we have
 $$
 P(a_f(p^{n-1})) > n ~\tx{exp}\( \frac{\log n}{104\log\log n}\)
 $$
 for all sufficiently large $n$ depending on $f$ and $p$.
 
Let $a>1$ be any natural number. Murty and S\'{e}guin \cite{MS} proved that if the exponent 
of $p$ in $a^{p-1}-1$ is bounded by a fixed constant for all primes~$p$, then there exists 
a positive constant $C$ depending on $a$ such that
\begin{equation}\label{PW}
P(\Phi_n(a)) ~>~ C \varphi(n)^2, ~ \forall ~n \in \N.
\end{equation}

In this section, we will deduce a number field analogue of \eqref{PW} and 
as a corollary we will derive a conditional lower bound for $P(a_f(p^n))$.

\begin{lem}\label{Pphin}
Let $A, B$ be non-zero coprime algebraic integers such that $A/B$ is not a root of unity. 
If there exists a constant $r$ such that for any prime ideal $\fP$ of the ring of integers 
of $\Q(A/B)$, 
\begin{equation}\label{nupr}
\nu_{\fP}\(\(\frac{A}{B}\)^{\mathcal{N}_0(\fP)-1} -1\) \leq r,
\end{equation}
then we have
\begin{equation*}
P(\Phi_n(A, B)) > \frac{h(A/B)}{13 r d_2} \cdot \frac{\varphi(n)^2}{d_1^{\w(n)+1}}
\end{equation*}
for all sufficiently large $n$ (depending on $A, B$). Here $d_1 =[\Q(A,B) : \Q]$, 
$d_2 = [\Q(A,B): \Q(A/B)]$, $h(\cdot)$ denotes the absolute logarithmic height of an algebraic 
number and $\cN_0$ denotes the absolute norm on $\Q(A/B)$.
\end{lem}

\begin{proof}
Let $K=\Q(A, B)$, $K_0=\Q(A/B)$. Also let $\cO_K$ (resp. $\cO_{K_0}$) be the ring 
of integers of $K$ (resp. $K_0$). Further $\cN_K, ~\cN_0$ denote the 
absolute norms on $K$ and $K_0$ respectively.
Using \lemref{impr}, for all sufficiently large $n$, we have the following factorization of $\Phi_n(A,B)\cO_K$,
\begin{equation}\label{fPhi}
\Phi_n(A,B)\cO_K = \fn \prod_{\substack{\fp | \Phi_n(A,B)\cO_K \\ \fp \nmid n\cO_K}} \fp^{\nu_{n,\fp}},
\end{equation}
where $\fn$ is an integral ideal of $K$ such that $\fn \div n\cO_K$ and 
$\fp \subset \cO_K$ varies over primitive divisors of $A^n-B^n$. Here 
$\nu_{n,\fp} = \nu_{\fp}(\Phi_n(A,B))$. Then we have 
\begin{equation*}
\cN_K(\fp) \equiv 1 (\mod n) \phantom{mm}\text{and}\phantom{mm} \nu_{n,\fp} =\nu_{\fp}(A^n-B^n).
\end{equation*}
We claim that $\nu_{\fp}(A^n-B^n) =\nu_{\fp}(A^{\cN_K(\fp)-1} - B^{\cN_K(\fp)-1})$. 

Let $R= \frac{\cN_K(\fp)-1}{n}$. Write $A^{\cN_K(\fp)-1} = \(B^n + (A^n-B^n)\)^R$ 
and then expand the right hand side by using the Binomial theorem, we get
\begin{equation}
A^{\cN_K(\fp)-1} = B^{\cN_K(\fp)-1} + \binom{R}{1} B^{n(R-1)} (A^n-B^n)
+ \cdots + \binom{R}{R} (A^n-B^n)^R.
\end{equation}
Hence
\begin{equation}
\frac{A^{\cN_K(\fp)-1} - B^{\cN_K(\fp)-1}}{A^n-B^n} = R B^{n(R-1)}
+ \binom{R}{2} B^{n(R-2)}(A^n-B^n)+ \cdots +  (A^n-B^n)^{R-1}.
\end{equation}
This implies that
\begin{equation*}
	\nu_{\fp}\(\frac{A^{\cN_K(\fp)-1} - B^{\cN_K(\fp)-1}}{A^n-B^n}\)=0,
\end{equation*}
since $\nu_{\fp}(RB)=0$ and $\nu_{\fp}(A^n-B^n) \geq 1$. This proves our claim.
We also note that
\begin{equation*}
\nu_{n,\fp} = \nu_{\fp} \(\(\frac{A}{B}\)^{\cN_K(\fp)-1}-1\).
\end{equation*}
Let $\fP= \fp \cap \cO_{K_0}$ and $s,t$ denote the ramification index and 
inertia degree of $\fp$ over $\fP$. Then we have $\cN_K(\fp) = \cN_0(\fP)^t$ and
\begin{equation*}
\nu_{n,\fp} = \nu_{\fp} \(\(\frac{A}{B}\)^{\cN_0(\fP)^t-1}-1\) = s~ \nu_{\fP} \(\(\frac{A}{B}\)^{\cN_0(\fP)^t-1}-1\).
\end{equation*}
 Arguing as in the previous claim, we can show that
\begin{equation*}
\nu_{\fP} \(\(\frac{A}{B}\)^{\cN_0(\fP)^t-1}-1\)=\nu_{\fP} \(\(\frac{A}{B}\)^{\cN_0(\fP)-1}-1\).
\end{equation*}
Thus we get
\begin{equation*}
\nu_{n,\fp} = s~ \nu_{\fP} \(\(\frac{A}{B}\)^{\cN_0(\fP)-1}-1\) ~\leq~ d_2~ \nu_{\fP} \(\(\frac{A}{B}\)^{\cN_0(\fP)-1}-1\),
\end{equation*}
where $d_2 = [K:K_0]$. By our assumption,
\begin{equation*}
\nu_{\fP} \(\(\frac{A}{B}\)^{\cN_0(\fP)-1}-1\) \leq r.
\end{equation*}
Hence we have $\nu_{n,\fp} \leq  rd_2$. Let $d_1 = [K:\Q]$ and 
\begin{equation*}
	M = \frac{h(A/B)}{13r d_2} \cdot \frac{\varphi(n)^2}{d_1^{\w(n)+1}}.
\end{equation*}
Consider the set
$$
S = \{ n :   P(\Phi_n(A,B)) \leq M\}.
$$
Suppose that $S$ is infinite. For any $n \in S$, applying \eqref{fPhi}, we have
\begin{equation}\label{N1}
\log|\cN_K(\Phi_n(A,B))|  \leq d_1\log n  ~+~  r d_2 \sum_{p \leq M}
\sum_{\substack{\fp \div p\cO_K\\ \cN_K(\fp) \equiv 1 (\mod n)}} \log \cN_K(\fp).
\end{equation}
The second sum on the right hand side can be written as
\begin{align}
\nonumber
\sum_{p \leq M}\sum_{\substack{\fp \div p\cO_K\\ \cN_K(\fp) \equiv 1 (\mod n)}} \log \cN_K(\fp) 
& ~=~
\sum_{p \leq M} \sum_{u=1}^{d_1} 
\sum_{\substack{\fp  \\~ \cN_K(\fp)=p^u\\ \cN_K(\fp) \equiv 1 (\mod n)}} \log \cN_K(\fp)\\ \nonumber
&~=~  \sum_{u=1}^{d_1} 
\sum_{\substack{p \leq M\\ p^u \equiv 1 (\mod n)}} \log(p^u) \sum_{\substack{\fp \\ \cN_K(\fp)=p^u}} 1 \\ \label{N2}
&~\leq~
  \sum_{u=1}^{d_1}  
  \sum_{\substack{p \leq M\\ p^u \equiv 1 (\mod n)}} d_1\log p.
\end{align}
We know that the congruence
\begin{equation}\label{cong}
X^u \equiv 1 (\mod n)
\end{equation}
has at most $2u^{\w(n)}$ solutions modulo $n$.
For all sufficiently large $n \in S$  (depending on $A, B$) and for any such solution $a$ modulo $n$, 
applying Brun-Titchmarsh inequality \eqref{BT}, 
we have
\begin{equation}
\begin{split}
 \sum_{\substack{p \leq M\\ p \equiv a (\mod n)}} \log p 
 ~\leq~ \frac{3M\log M}{\varphi(n) \log (M/n)}.
\end{split}
\end{equation}
Hence
\begin{equation}\label{N3}
 \sum_{\substack{p \leq M\\ p^u \equiv 1 (\mod n)}} 
 ~\log p~ 
 ~\leq~
  2u^{\w(n)} \frac{3M\log M}{\varphi(n) \log (M/n)}.
\end{equation}
Thus from \eqref{N1}, \eqref{N2} and \eqref{N3}, we obtain
\begin{equation}
\log|\cN_K(\Phi_n(A,B))| ~\leq~  d_1 \log n + 6  r d_2 d_1^{\w(n)+2} \frac{M\log M}{\varphi(n) \log (M/n)}.
\end{equation}
Now using \lemref{Nphin}, we get
\begin{equation}\label{c}
 h(A/B) d_1 \varphi(n) (1 + o(1)) ~\leq~ d_1 \log n + \frac{6}{13}d_1h(A/B) \varphi(n) \frac{\log M}{\log (M/n)}.
\end{equation}
Since $\frac{\log M}{\log (M/n)}$ tends to $2$ as $n \in S$ tends to infinity, we get
a contradiction for sufficiently large~$n$. This implies that
\begin{equation*}
 P(\Phi_n(A,B)) ~>~ \frac{h(A/B)}{13r d_2} \cdot \frac{\varphi(n)^2}{d_1^{\w(n)+1}}
\end{equation*}
for all sufficiently large $n$.
\end{proof}

\begin{rmk}
Let the notations be as in \lemref{Pphin}. We define 
$$
\widetilde{P}_K(\Phi_n(A,B)) = \max\{ \cN_K(\fp) : \fp \subset \cO_K,~\fp ~|~\Phi_n(A,B)\cO_K\}.
$$
Under the assumption \eqref{nupr}, arguing as in the proof of \lemref{Pphin}, we can deduce that
\begin{equation}
\widetilde{P}_K(\Phi_n(A,B)) > \frac{h(A/B)}{7 r d_2} \varphi(n)^2
\end{equation}
for all sufficiently large $n$ depending on $A$ and $B$.
\end{rmk}

\subsection{Proof of \thmref{pafp}}
Let $f$ be a normalized cuspidal Hecke eigenform of even weight $k \geq 2$ and level $N$ with
Fourier coefficients $\{a_f(n) : n \in \N\}$. As before, for any rational prime $p \nmid N$ and 
natural number $n \ge 1$, we have
\begin{equation}\label{luc}
a_f(p^{n-1}) = \frac{\a_p^n - \b_p^n}{\a_p-\b_p},
\end{equation}
where $\a_p, \b_p$ are the roots the polynomial $x^2 - a_f(p)x + p^{k-1}$.
Let $\gamma_p = \frac{\alpha_p}{\beta_p}$.

First suppose that Fourier coefficients of $f$ are rational integers, $p > 3, ~p \nmid N$
and $\gamma_p$ is not a root of unity.
As before, let us define $\nu_{f,p} = \nu_p(a_f(p))$. 
By Deligne's bound, we have $\nu_{f,p} \leq k/2-1$. Set
 
\begin{equation*}
A_p = \frac{\a_p}{p^{\nu_{f,p}}} \phantom{mm}\text{and}\phantom{mm} B_p = \frac{\b_p}{p^{\nu_{f,p}}}.
\end{equation*}
Then $A_p, B_p$ are roots of the polynomial 
$x^2-a_f(p) p^{-\nu_{f,p}}x  +  p^{k-1-2\nu_{f,p}}$ in $\Z[x]$ and $(A_p, B_p)=1$. Also we have 
$$
\Q(A_p,B_p)=\Q(\a_p)=\Q(\gamma_{p}), \phantom{m} 
[\Q(\a_p) :\Q]=2 \phantom{m}
\text{and} \phantom{m} 
h(\gamma_{p}) = \(\frac{k-1}{2}-\nu_{f,p}\)\log p.
$$
From \eqref{luc}, for any integer $n>1$, we have
\begin{equation}
a_f(p^{n-1}) 
= p^{(n-1)\nu_{f,p}}\frac{A_p^n - B_p^n}{A_p-B_p}
=  p^{(n-1)\nu_{f,p}} \prod_{\substack{t | n \\ t>1}} \Phi_t(A_p, B_p).
\end{equation}
Hence we get $P(a_f(p^{n-1})) \geq P(\Phi_n(A_p, B_p))$. Applying \lemref{Pphin},  we have
\begin{equation}
P(\Phi_n(A_p,B_p)) ~>~ \frac{(k-1-2\nu_{f,p})\log p}{52r} \cdot \frac{\varphi(n)^2}{2^{\w(n)}}
\end{equation}
for all sufficiently large $n$ (depending on $f,p$).

Now let us consider the case when Fourier coefficients are not necessarily rational integers. As before,
let $\K_f= \Q(\{a_f(n) : n \in \N\}), d_f = [\K_f : \Q]$. Also let $h_{f,p}$ be the class number of $\Q(\a_p)$. 
There exists an extension $\L_{f,p}$ of $\Q(\a_p)$ of 
degree at most $h_{f,p}$ over $\Q(\a_p)$ such that the ideal $(\a_p, \b_p)$
 is principal in $\L_{f,p}$. 
 Also let $(\a_p, \b_p)=(\eta_p)$ in $\L_{f,p}$. Set
\begin{equation*}
A_p = \frac{\a_p}{\eta_p} \phantom{mm}\text{and}\phantom{mm} B_p = \frac{\b_p}{\eta_p}.
\end{equation*}
Then we have $(A_p, B_p)=1$. For any integer $n>1$, we have
\begin{equation}
a_f(p^{n-1}) = \eta_{p}^{n-1} \frac{A_p^n - B_p^n}{A_p-B_p}= \eta_{p}^{n-1} \prod_{\substack{t | n\\ t>1}} \Phi_t(A_p, B_p).
\end{equation}
Hence we get $P(a_f(p^{n-1})) \geq P(\Phi_n(A_p,B_p))$. Using \lemref{Pphin},  we get
\begin{equation}
 P(\Phi_n(A_p,B_p)) > \frac{h(\gamma_{p})}{13 r d_2} \cdot \frac{\varphi(n)^2}{d_1^{\w(n)+1}},
\end{equation}
for all sufficiently large $n$ (depending on $f,p$). Here $d_1 = [\Q(A_p, B_p) : \Q] \leq 2d_f h_{f,p}$ 
and $d_2 = [\Q(A_p, B_p) : \Q(\gamma_{p})] \leq 2 h_{f,p}$. Thus we get
\begin{equation*}
P(\Phi_n(A_p,B_p)) ~>~  \frac{h(\gamma_{p})}{26 r h_{f,p} } 
\cdot  \frac{\varphi(n)^2}{(2d_f h_{f,p})^{\w(n)+1}}
\end{equation*}
for all sufficiently large $n$ (depending on $f,p$).

\section*{Acknowledgments}
The authors would like to thank Ram Murty, Purusottam Rath and the referee
for valuable suggestions. The authors would also like to thank SPARC project 445 for partial
financial support and Indo French program in Mathematics (IFPM). 
The second and the third author would also like to 
acknowledge the support of DAE number theory plan project.


\begin{thebibliography}{100}
\bibitem{BLGHT11}
T.~Barnet-Lamb, D.~Geraghty, M.~Harris and R.~Taylor,
{\em A Family of Calabi-Yau Varieties and
Potential Automorphy II}, 
Publ. Res. Inst. Math. Sci. \textbf{47} (2011), no. 1, 29--98. 


\bibitem{BGPS}
M. A. Bennett, A. Gherga, V. Patel and S. Siksek,
{\em Odd values of the Ramanujan tau function},
Math. Ann. \textbf{382} (2022), no. 1-2, 203--238.


\bibitem{BGH}
Y. F. Bilu, S. Gun and H. Hong, 
{\em Uniform explicit  Stewart's theorem on prime factors 
of linear recurrences}, arXiv:2108.09857 (2021).


\bibitem{YB}
Y. Bugeaud,
{\em On the greatest prime factor of $ax^m+by^n$, II},
Bull. London Math. Soc. \textbf{32} (2000), no. 6, 673--678.


\bibitem{CHT08}
L. ~Clozel, M.~Harris and R.~Taylor, 
{\em Automorphy for some $\ell$-adic lifts of 
automorphic mod $\ell$ Galois representations}, 
With Appendix A, summarizing unpublished work 
of Russ Mann, and 
Appendix B by Marie-France Vign\'eras, 
Inst. Hautes \'Etudes Sci. Publ. Math. 
 \textbf{108} (2008), 1--181. 
 
 
\bibitem{CDP}
R. Crandall, K. Dilcher and C. Pomerance, 
{\em A search for Wieferich and Wilson primes}, 
Math. Comp. \textbf{66} (1997), no. 217, 433--449.


\bibitem{De}
P. Deligne,
Formes modulaires et repr\'{e}sentations $\ell$-adiques,
{\em S\'{e}minaire Bourbaki}. Vol. 1968/69: Expos\'{e}s 347–363, 
Exp. No. 355, 139--172,
Lecture Notes in Math., 175, Springer, Berlin, 1971.


\bibitem{ED}
E. Dobrowolski, {\em On a question of Lehmer and the number 
of irreducible factors of a polynomial}, 
Acta Arith. \textbf{34} (1979), no. 4, 391--401.


\bibitem{GGK}
M. Z. Garaev, V. C. Garcia and S. V. Konyagin,
{\em A note on the Ramanujan $\tau$-function},
Arch. Math. (Basel) \textbf{89} (2007), no. 5, 411--418.


\bibitem{GM}
S. Gun and M. R. Murty,
{\em Divisors of Fourier coefficients of modular forms},
 New York J. Math. \textbf{20} (2014), 229--239.

\bibitem{HR}
H. Halberstam and H.-E. Richert,
Sieve Methods, London Mathematical Society Monographs, No. 4,
{\em Academic Press [Harcourt Brace Jovanovich, Publishers]}, 
London-New York, 1974.

\bibitem{HSBT10}
M.~Harris, N.~Shepherd-Barron, R.~Taylor,
{\em A family of Calabi-Yau varieties and potential automorphy}, 
Ann. of Math. \textbf{171} (2010), no. 2, 779--813.


\bibitem{HH}
H. Hong,
{\em Stewart's Theorem revisited: suppressing the norm $\pm1$ hypothesis}, 
arXiv: 2204.01858 (2022).


\bibitem{KRW}
E. Kowalski, O. Robert and J. Wu, 
{\em Small gaps in coefficients of $L$-functions and $\mathcal{B}$-free 
numbers in short intervals}, Rev. Mat. Iberoam. \textbf{23} (2007), no. 1, 281--326.


\bibitem{LO}
J. C. Lagarias and A. M. Odlyzko,
{\em Effective versions of the Chebotarev density theorem},
Algebraic number fields: L-functions and Galois properties (Proc. Sympos., 
Univ. Durham, Durham, 1975), Academic Press, London, 1977, 409--464.


\bibitem{LS}
F. Luca and I. E. Shparlinski, 
{\em Arithmetic properties of the Ramanujan function},
Proc. Indian Acad. Sci. Math. Sci. \textbf{116} (2006), no.1, 1--8.


\bibitem{MMprime}
M. R. Murty and V. K. Murty,
{\em Prime divisors of Fourier coefficients of modular forms},
Duke Math. J. \textbf{51} (1984), no. 1, 57--76.

\bibitem{MModd}
M. R. Murty and V. K. Murty,
{\em Odd values of Fourier coefficients of certain modular forms},
Int. J. Number Theory, \textbf{3} (2007), no. 3, 455--470.

\bibitem{MMS}
M. R. Murty, V. K. Murty and N. Saradha,
{\em Modular forms and the Chebotarev density theorem},
Amer. J. Math. \textbf{110} (1988), no. 2, 253--281.

\bibitem{MS}
M. R. Murty and F. S\'{e}guin,
{\em Prime divisors of sparse values of cyclotomic polynomials and Wieferich primes},
J. Number Theory \textbf{201} (2019), 1--22.

\bibitem{Rib}
K. A. Ribet,
{\em Galois representations attached to eigenforms with Nebentypus},
Modular functions of one variable, V (Proc. Second Internat. Conf., Univ. Bonn, Bonn, 1976), 
Lecture Notes in Math., Vol. \textbf{601}, Springer, Berlin, 1977, 17--51.

\bibitem{Sc}
A. Schinzel, 
{\em Primitive divisors of the expression $A^n-B^n$ in algebraic number fields},
J. Reine Angew. Math. \textbf{268 (269)} (1974), 27--33.

\bibitem{JPS}
J.-P. Serre,
{\em Divisibilit\'{e} de certaines fonctions arithm\'{e}tiques},
Enseign. Math. (2) \textbf{22} (1976), no. 3-4, 227--260.

\bibitem{Sr}
J.-P. Serre, 
{\em Quelques applications du th\'{e}or\`{e}me de densit\'{e} de Chebotarev},
Inst. Hautes \'Etudes Sci. Publ. Math. \textbf{54} (1981), 323--401.

\bibitem{Sh}
G. Shimura,
Introduction to the arithmetic theory of automorphic functions,
 Publications of the Mathematical Society of Japan, No. 11. Iwanami Shoten, Publishers, 
 {\em Tokyo; Princeton University Press}, Princeton, N.J., 1971.

\bibitem{St}
C. L. Stewart,
{\em On divisors of Lucas and Lehmer numbers},
Acta Math. \textbf{211} (2013), no. 2, 291--314.

\bibitem{TZ}
J. Thorner and A. Zaman,
{\em A unified and improved Chebotarev density theorem},
Algebra Number Theory \textbf{13} (2019), no. 5, 1039--1068.

\bibitem{PV}
P. M. Voutier,
{\em An effective lower bound for the height of algebraic numbers},
Acta Arith. \textbf{74} (1996), no. 1, 81--95.
\end{thebibliography}
\end{document}